\documentclass[12pt]{amsart}
\pdfoutput=1
\usepackage{amssymb, amsmath}
\usepackage{graphicx} 
\usepackage{textcomp}
\usepackage{blkarray}
\usepackage{easybmat}
\usepackage{multirow}

\newtheorem{theorem}{Theorem}[section]
\newtheorem*{main_cor}{Corollary \ref{main_cor}}
\newtheorem*{main}{Theorem \ref{main}}

\newtheorem{lemma}[theorem]{Lemma}
\newtheorem{corollary}[theorem]{Corollary}

\theoremstyle{definition}		

\title{Rank gradient of small covers}
\author{Darlan Gir\~ao}

\begin{document}
\maketitle

\begin{abstract} We prove that if $M\longrightarrow P$ is a small cover of a compact right-angled hyperbolic polyhedron then $M$ admits a cofinal tower of finite sheeted covers with positive rank gradient. As a corollary, if $\pi_1(M)$ is commensurable with the reflection group of $P$, then    $M$ admits a cofinal tower of finite sheeted covers with positive rank gradient. 
\end{abstract}

\section{Introduction}\label{section_intro}

Let ￼ $P^n$ be an $n$￼-dimensional \textit{simple convex polytope}. Here $P^n$ ￼ is simple if the number of codimension-one faces  meeting at each vertex is $n$￼. Equivalently, the dual ￼$K_P$ ￼ of its boundary complex $\partial P^n$ is an $(n-1)$-dimensional simplicial sphere. A \textit{Small Cover} of $P^n$ is  an $n$￼-dimensional manifold endowed with an action of the group $\mathbb{Z}_2^n$ whose orbit space is $P^n$. The notion of small cover was introduced and studied by Davis and Januszkiewicz in \cite{DJ}. We will be dealing mostly with 3-dimensional polytopes. In the case    $P$ is a compact right-angled polyhedron in $\mathbb{H}^3$ then  Andreev's theorem \cite{An}  implies that all vertices have valence three and in particular $P$ is a simple convex polytope.  

Let $G$ be a finitely generated group. The \textit{rank of $G$} is the minimal  number of elements needed to generate $G$, and is denoted by  $\text{rk}(G)$.  If $G_j$ is a finite index subgroup of $G$, the Reidemeister-Schreier process \cite{LS} gives an upper bound on the rank of $G_j$.	 
$$\text{rk}(G_j)-1\leq [G:G_j](\text{rk}(G)-1)$$  
Recently Lackenby introduced the notion of \textit{rank gradient} \cite{La}.
 Given a finitely generated group $G$ and a collection $\{G_j\}$ of finite index subgroups,  the \textit{rank gradient} of  the pair $(G,\{G_j\})$ is defined by 
$$\text{rgr}(G,\{G_j\})=\lim_{j\rightarrow\infty}\frac{\text{rk}(G_j)-1}{[G:G_j]}$$   
We say that the  collection of finite index subgroups $\{G_j\}$ is \textit{cofinal} if $\cap_j G_j=\{1\}$, and we call it a  \textit{tower} if $G_{j+1}<G_j$.

In general it is very hard to construct co-final families $(G,\{G_j\})$ with positive rank gradient. For instance, it seems that  only  recently the first examples of torsion-free finite covolume Kleinian groups with this property were given in \cite{Gi}. Before stating  main result we need some terminology. 

If $M$ is a finite volume hyperbolic 3-manifold, we call the family of covers $\{M_j\longrightarrow M\}$ \textit{cofinal} (resp. a \textit{tower}) if $\{\pi_1(M_j)\}$ is cofinal (resp. a tower). By rank gradient of the the pair $(M,\{M_j\})$,  $\text{rgr}(M,\{M_j\})$, we mean the rank gradient of $(\pi_1(M),\{\pi_1(M_j)\})$.

\begin{theorem}\label{main}
Let $M\longrightarrow P$ be a small cover of a compact, right-angled hyperbolic polyhedron of dimension $3$. Then $M$ admits a cofinal tower of finite sheeted covers $\{M_j\longrightarrow M\}$ with positive rank gradient.
\end{theorem} 

We remark here that this is not true for 3-dimensional polytopes in general.  Let $T^3\longrightarrow C$ be the  covering  of a cube in 3-dimensional Euclidean space by the 3-torus $T^3$. It is easy to see that any subgroup of $\pi_1(T^3)=\mathbb{Z}^3$ has bounded rank and therefore rank gradient with respect to any tower of covers is zero. 

This theorem has the following consequence
\begin{corollary}\label{main_cor}
Let $M$ be a finite volume hyperbolic $3$-manifold such that $\pi_1(M)$ is commensurable with the group generated by reflections along the faces of  a compact, right-angled hyperbolic polyhedron $P\subset\mathbb{H}^3$. Then $M$ admits a cofinal tower of finite sheeted covers $\{M_j\longrightarrow M\}$ with positive rank gradient.
\end{corollary}

We note that this corollary is complementary to the results of \cite{Gi}, where ideal right-angled polyhedra were considered. The key idea there was to estimate the rank of the fundamental group of the manifolds by estimating their number of cusps. Here the estimates on the rank of the fundamental groups are given in terms of the rank of the $\mod$ 2  homology. 

The study of the rank of the fundamental group of (finite volume hyperbolic) 3-manifolds has always been a central theme in low dimensional topology. In recent years the study of the rank gradient for this class of groups has received special attention. For instance,  motivated by the seminal paper \cite{La1} of Lackenby, Long--Lubotzky--Reid \cite{LLR} prove that every finite volume hyperbolic 3-manifold has   a cofinal tower of covers in which the \textit{Heegaard genus}  grows linearly with the degree of the covers.  Whether or not the same happens to the rank of their fundamental groups is a major open problem. Another important recent work using these notions is \cite{AN}. There they connect the problem related to the growth of the rank of $\pi_1$ and the growth of the Heegaard genus in a cofinal tower of  hyperbolic 3-manifolds to a problem in topological dynamics, \textit{the fixed price problem} (see \cite{Fa} and \cite{Ga}). These papers have all been motivation for the current work. 
 
\section*{Acknowledgements}
The author would like to thank Alan Reid for suggesting him to look at  the  paper \cite{DJ}. He would also like to thank The Erwin Schrodinger International Institute for Mathematical Physics for the hospitality and financial support during his visit in the occasion of the workshop Golod-Shafarevich Groups and Algebras and Rank Gradient. Finally, he would like to thank the organizers of this event.

\section{Small covers}
Recall that an $n$-dimensional  convex polytope $P^n$ is simple if the number of codimension-one faces  meeting at each vertex is $n$. Equivalently, the dual ￼$K_P$ ￼ of its boundary complex $\partial P$ is an $(n-1)$-dimensional simplicial complex. A Small Cover of $P$ is  an $n$￼-dimensional manifold endowed with an action of the group $\mathbb{Z}_2^n$ whose orbit space is $P$. 

Let $K$ be a finite simplicial complex of dimension $n-1$.  For $0\leq i\leq n-1$, let $f_i$ be the number of $i$-simplices of  $K$. Define a polynomial $\Phi_K(t)$ of degree $n$ by 
$$\Phi_K(t)=(t-1)^n+\sum_{i=0}^{n-1}f_i(t-1)^{n-1-i}$$
and let $h_i$ be the coefficient of $t^{n-i}$ in this polynomial, i.e., 
$$\Phi_K(t)=\sum_{i=0}^{n}h_it^{n-i}$$ 

If we restrict to the case where $K$ is the dual $K_P$ of the boundary complex of a convex simple polytope $P^n$,  then one can see that $f_i$ is the number of faces of $P^n$ of codimension $i+1$.  Let $h_i(P^n)$ denote the coefficient of $t^{n-i}$ in $\Phi_{K_P}(t)$ 

One of the main results of \cite{DJ}, here stated in a very particular setting,  is

\begin{theorem}\label{DJthm}
Let $\pi:M^n\longrightarrow P^n$ be a small cover of a simple convex polytope $P^n$ and let  $b_i(M^n,\mathbb{Z}_2)$  be the $i^{\text{th}}$ $\mod 2$ Betti number of $M^n$. Then $b_i(M^n,\mathbb{Z}_2)=h_i(P^n)$
\end{theorem} 

As observed in \cite{DJ}, it  is somewhat surprising  that all $\mod 2$ Betti numbers  of a small cover $M^n$ depend on $P^n$ only. They also show that this theorem does not hold for  homology groups in general. They provide  small covers of a square $Q$ by   tori and a Klein bottles are such that the rational Betti numbers are not determined by $Q$.  

When $P$ is a right-angled dodecahedron in $\mathbb{H}^3$ then \cite{GS} shows that up to homeomorphism there exists exactly $25$  small covers of $P$.   \cite{Ch} estimates the number of orientable small covers of the $n$-dimensional cube.  Also, if $P$ is a  $3$-dimensional convex polytope,  \cite{NN} proves that $P$ admits an orientable small cover.  They also prove that unless $P$ is a $3$-simplex, then it admits a non-orientable small cover. 

\section{Proof of theorem}
In this section we prove

\begin{main} Let $M\longrightarrow P$ be a small cover of a compact, right-angled hyperbolic polyhedron of dimension $3$. Then $M$ admits a cofinal tower of finite sheeted covers $\{M_j\longrightarrow M\}$ with positive rank gradient.
\end{main}

\begin{proof}
As observed above, when  $P$ be a compact right-angled polyhedron in $\mathbb{H}^3$ then  Andreev's theorem (\cite{An})  implies that all vertices have valence three and in particular $P$ is a simple convex polytope. 
Let $V, E$ and $F$ denote the number of vertices, edges and faces, respectively, of a $3$-dimensional simple polyhedron $P$. Straightforward computations show that 
$$\Phi_{K_P}(t)=t^3+(F-3)t^2+(3-2F+E)t+(V-E+F-1)$$
and thus $h_0(P)=1$, $h_1(P)=F-3$, $h_2(P)=3-2F+E$ and $h_3(P)=V-E+F-1$. Since $P$ is simple we also have $E=3V/2$. And since $V-E+F=2$ ($\partial P$ is topologically a sphere) this gives $F=(1/2)V+2$ and  therefore  $h_1(P)=(1/2)V-1$. 

The strategy involved in the proof is similar to the proof of the main theorem in \cite{Gi}. Given $P\in\mathbb{H}^3$,  construct a family of polyhedra $$P=P_0, P_1, \ldots, P_j,\ldots$$ such that $P_{j+1}$ is obtained from $P_j$ by reflecting $P_j$ along one of its faces. This must be done in a way such that the following holds: if $x\in\mathbb{H}^3$, then there exists $j$ sufficiently large so that $x$ lies in the interior of $P_j$. This means that the family $\{P_j\}$ is an exhaustion of $\mathbb{H}^3$. Denote by $G_j$ the group generated by reflections along the faces of $P_j$. If the family $\{P_j\}$ is constructed as above, then it is easy to see that $G_{j+1}<G_j$ (with index 2) and it can be shown that the tower $\{G_j\}$ is cofinal (see \cite{Ag}). We refer the reader to   \cite{Gi} for a detailed proof of this fact.      

Now let $M\longrightarrow P$ be a small cover of $P$, and let $M_j\longrightarrow M$ be the cover corresponding to the group $\pi_1(M)\cap G_j$. Recall that the degree of the cover $M\longrightarrow P$ is $2^3$.

\begin{lemma}
$[\pi_1(M_j):\pi_1(M_{j+1})]=2$
\end{lemma} 
\begin{proof}[Proof of lemma]
First observe that $[G_j:G_{j+1}]=2$. Since $\pi_1(M_1)=G_1\cap\pi_1(M)$,  we must have $[\pi_1(M):\pi_1(M_1)]\leq2$. If this index were $1$, then it would mean that $\pi_1(M_1)=\pi_1(M)<G_1$ from which would follow that $M_1$ is a manifold cover of the simple polyhedron $P_1$ of degree $2^2$. But this is not possible, since any manifold cover of a $3$-dimensional simple polyhedron must have degree at least $2^3$ (see \cite{DJ}, \cite{GS}).  The remaining cases follow by induction.  
\end{proof}

Since $[G_j:G_{j+1}]=2$, from the above lemma and an inductive argument  we see that $M_j\longrightarrow P_j$ is a cover of degree $2^3$. In particular this implies that $M_j$ is a small cover of $P_j$.  From theorem \ref{DJthm} we  have 
$$ b_1(M_j,\mathbb{Z}_2)=h_1(P_j)$$
Denote by $V_j$ the number of vertices of $P_j$. From   the computations of $h_1$,  $$b_1(M_j, \mathbb{Z}_2)=h_1(P_j)=\frac{V_j}{2}-1$$
Also note that a lower bound for $\text{rk}(\pi_1(M_j))$ is $b_1(M_j,\mathbb{Z}_2)$ and thus
$$\text{rk}(\pi_1(M_j))\geq\frac{V_j}{2}-1$$
We also have $[\pi_1(M):\pi_1(M_j)]=2^j$. Therefore 
$$\text{rgr}(\pi_1(M)),\{\pi_1(M_j)\})=\lim_{j\rightarrow\infty}\frac{\text{rk}(\pi_1(M_j))-1}{[\pi_1(M):\pi_1(M_j)]}\geq\lim_{j\rightarrow\infty}\frac{V_j-3}{2^{j+1}}$$ 
We thus need to show that $V_j$ is of magnitude $2^j$.  This is a consequence of a theorem of Atkinson \cite{At}.

\begin{theorem}[Atkinson] There exist constants $C,D>0$ such that if $P$ is a compact right-angled polyhedron in $\mathbb{H}^3$ with $V$ vertices then  
$$C(V-8)\leq\text{vol}(P)\leq D(V-10)$$
\end{theorem} 

We now observe that, in our setting, $\text{vol}(P_j)=2^j\text{vol}(P)$ and thus $$D(V_j-10)\geq 2^j\text{vol}(P)\geq 2^jC(V-8)$$ which gives $$V_j\geq 2^j\frac{C}{D}(V-8)+10$$ where $V$ is the number of vertices in $P$. 
Also, the second inequality in Atkinson's theorem provide $V >8$.  The theorem is now proved.
\end{proof}


\section{Extending the examples}

Theorem \ref{main} has an interesting corollary, which complements the family of manifolds provided in \cite{Gi}. 

\begin{main_cor}
Let $N$ be a closed hyperbolic $3$-manifold such that $\pi_1(N)$ is commensurable with the group generated by reflections along the faces of  a compact, right-angled hyperbolic polyhedron $P\subset\mathbb{H}^3$. Then $N$ admits a cofinal tower of finite sheeted covers $\{N_j\longrightarrow N\}$ with positive rank gradient.
\end{main_cor}

\begin{proof}
First we note that, by passing to a finite cover, we may assume $N$ is orientable.  Note also that \cite{NN} implies orientable small covers of $P$ exist and therefore $N$ is commensurable with a small cover $M\longrightarrow P$. Let $N'$ be the manifold cover of both $M$ and $N$ corresponding to the group $\pi_1(M)\cap\pi_1(N)$. Consider now $N_j\longrightarrow N$ corresponding to the group $\pi_1(N')\cap G_j$, where the family  $\{G_j\}$ is given as in the proof of theorem \ref{main}.   Consider also $\{M_j\}$ the tower where $M_j$ is a small cover of $P_j$, also as in the proof of theorem \ref{main}. 

Note that $\pi_1(N_j)=\pi_1(N')\cap G_j=\pi_1(N')\cap\pi_1(M_j)<\pi_1(M_j)$ and therefore we have the following diagram of covers, where the labels in the arrows indicate the degree of the cover. 

$$\begin{array}[c]{cccccccccccccc}
&&P             & \stackrel{2}\leftarrow&    P_1      &\stackrel{2}\leftarrow &\cdots&   \stackrel{2}\leftarrow  &       P_j     &\stackrel{2}\leftarrow \cdots \\
&&\uparrow\scriptstyle{2^3}&                 & \uparrow\scriptstyle{2^3}&                &           &                    & \uparrow\scriptstyle{2^3}&\\
&&M            &\stackrel{2}\leftarrow &    M_1     &\stackrel{2}\leftarrow &\cdots &  \stackrel{2}\leftarrow  &      M_j     &\stackrel{2}\leftarrow  \cdots\\
&&\uparrow&                 & \uparrow&                &           &                    & \uparrow&\\
N&\leftarrow &N'            &\leftarrow  &    N_1     &\leftarrow &\cdots &\leftarrow     &      N_j     &\leftarrow\cdots  \\
\end{array}$$

Agol-Culler-Shallen proved  the following in \cite{ACS} (see also \cite{Sh}).

\begin{theorem}
Let $M$ be a closed, orientable hyperbolic $3$ –manifold such that $b_1(M, \mathbb{Z}_p)=r$ for a given prime $p$. Then for any finite sheeted covering space $M'$ of $M$,  $b_1(M', \mathbb{Z}_p)\geq r-1$.  
\end{theorem}
  
  We thus have $$\text{rk}(\pi_1(N_j))\geq b_1(N_j,\mathbb{Z}_2)\geq b_1(M_j,\mathbb{Z}_2)-1=\frac{V_j}{2}-2$$ and therefore all we need to do is show that $[\pi_1(N):\pi_1(N_{j})]$ grows at most as fast as $2^j$. But from the above diagram we see that $[\pi_1(N_j):\pi_1(N_{j+1})]\leq 2$ and we are done.
\end{proof}

\vspace{.4cm}
\noindent
\address{\textsc{Department of Mathematics,\\
Universidade Federal do Cear\'a}}\\
\email{\textit{E-mail:}\texttt{ dgirao@mat.ufc.br}}


\begin{thebibliography}{99}

\bibitem [AN]{AN} M. Ab\'ert, N. Nikolov, \textit{ Rank gradient, cost of groups and the rank versus Heegaard genus problem}, J. of EMS, to appear. 
\bibitem  [Ag] {Ag} I. Agol, \textit{Criteria for virtual fibering},  J. Topol.  1  (2008),  no. 2, 269--284. 
\bibitem [ACS]{ACS} I. Agol, M. Culler, P. Shalen, \textit{Dehn surgery, homology and hyperbolic volume},  Algebr. Geom. Topol. 6 (2006), 2297–- 2312.
\bibitem [An]{An} E. M. Andreev, \textit{On convex polyhedra in Lobachevski spaces}, Math. USSR Sbornik 10 (1970), no. 3, 413--440.
\bibitem[At]{At} C. Atkinson, \textit{Volume estimates for equiangular hyperbolic Coxeter polyhedra}, Algebraic \& Geometric Topology 9 (2009) 1225--1254.
\bibitem[Ch]{Ch} S. Choi, \textit{The number of orientable small covers over cubes}, Proc. Japan Acad. Ser. A Math. Sci. Volume 86, Number 6 (2010), 97--100.
\bibitem [DJ]{DJ} M. Davis, T. Januszkiewicz, \textit{Convex polytopes, coxeter orbifolds and torus actions}, Duke Math. J. vol 62, no. 2 (1991), p. 417--451.
\bibitem[Fa]{Fa} M. Farber, \textit{Geometry of growth: approximation theorems for $L^2$ invariants}, Math. Ann. 311 (2) (1998).
\bibitem [Ga]{Ga} D. Gaboriau, \textit{Co\^ut des relations d`\'equivalence et des groupes}. (French) [Cost of equivalence relations and of groups] Invent. Math. 139 (2000), no. 1, 41--98.  
\bibitem [Gi]{Gi} D. Girao, \textit{Rank gradient in cofinal towers of certain Kleinian groups},  Groups, Geometry and Dyanmics, to appear. 
\bibitem [GS]{GS} A. Garrison, R. Scott, \textit{Small covers of the dodecahedron and the 120-cell}, Proceedings of the AMS, Volume 131, Number 3, 963-–971. 
\bibitem  [La] {La} M. Lackenby, \textit{Expanders, rank and graphs of groups}, Israel J. Math. 146 (2005) 357--370 
\bibitem  [La1] {La1} M. Lackenby, \textit{Heegaard splittings, the virtually Haken conjecture and Property ($\tau$)}, Invent. Math. 164 (2006) 317--359 
\bibitem [LLR]{LLR} D. Long, A. Lubotzky, A. Reid, \textit{Heegaard genus and Property ($\tau$) for hyperbolic 3-manifolds}, Journal of Topology 1 (2008) 152--158.
\bibitem [LS]{LS} R. Lyndon, P.Schupp, \textit{Combinatorial group theory}, Springer-Verlag, Berlin, 1977.
\bibitem [NN]{NN} H. Nakayama, Y. Nishimura, \textit{The orientability of small covers and coloring simple polytopes}, Osaka J. Math., 42 (2005), 243–-256
\bibitem [Sh]{Sh} P. Shalen, \textit{Hyperbolic volume, Heegaard genus and ranks of groups}, Geometry \& Topology Monographs 12 (2007) 335–-349.

\end{thebibliography}
\end{document}